\documentclass[reqno]{amsart}
\usepackage{amssymb}
\usepackage[colorlinks=true,hypertex]{hyperref}
\allowdisplaybreaks[4]
\theoremstyle{plain}
\newtheorem{thm}{Theorem}

\theoremstyle{remark}
\newtheorem{rem}{Remark}
\DeclareMathOperator{\td}{d\mspace{-2mu}}
\date{Commenced on 15 February 2009 and completed on 17 February 2009 in Melbourne}
\date{}

\begin{document}

\title[Sharpening and generalizations of Shafer's inequality]
{Sharpening and generalizations of Shafer's inequality for the arc tangent function}

\author[F. Qi]{Feng Qi}
\address[F. Qi]{Research Institute of Mathematical Inequality Theory, Henan Polytechnic University, Jiaozuo City, Henan Province, 454010, China}
\email{\href{mailto: F. Qi <qifeng618@gmail.com>}{qifeng618@gmail.com}, \href{mailto: F. Qi <qifeng618@hotmail.com>}{qifeng618@hotmail.com}, \href{mailto: F. Qi <qifeng618@qq.com>}{qifeng618@qq.com}}
\urladdr{\url{http://qifeng618.spaces.live.com}}

\author[B.-N. Guo]{Bai-Ni Guo}
\address[B.-N. Guo]{School of Mathematics and Informatics, Henan Polytechnic University, Jiaozuo City, Henan Province, 454010, China}
\email{\href{mailto: B.-N. Guo <bai.ni.guo@gmail.com>}{bai.ni.guo@gmail.com}, \href{mailto: B.-N. Guo <bai.ni.guo@hotmail.com>}{bai.ni.guo@hotmail.com}}
\urladdr{\url{http://guobaini.spaces.live.com}}

\begin{abstract}
In this paper, we sharpen and generalize Shafer's inequality for the arc tangent function. From this, some known results are refined.
\end{abstract}



\thanks{The first author was partially supported by the China Scholarship Council}

\thanks{This paper was typeset using \AmS-\LaTeX}

\maketitle

\section{Introduction and main results}

In~\cite{Shafer-arctan-amm-66}, the following elementary problem was posed: Show that for $x>0$
\begin{equation}\label{Shafer-arctan-ineq}
\arctan x>\frac{3x}{1+2\sqrt{1+x^2}\,}.
\end{equation}
In~\cite{Shafer-arctan-amm-solutaions-67}, the following three proofs for the inequality~\eqref{Shafer-arctan-ineq} were provided.
\begin{description}
  \item[Solution by Grinstein]
Direct computation gives
$$
\frac{\td F(x)}{\td x}=\frac{\bigl(\sqrt{1+x^2}\,-1\bigr)^2} {\bigl(1+x^2\bigr)\bigl(1+2\sqrt{1+x^2}\,\bigr)^2},
$$
where
$$
F(x)=\arctan x-\frac{3x}{1+2\sqrt{1+x^2}\,}.
$$
Now $\frac{\td F(x)}{\td x}$ is positive for all $x\ne0$, whence $F(x)$ is an increasing function. Since $F(0)=0$, it follows that $F(x)>0$ for $x>0$.
  \item[Solution by Marsh]
It follows from $(\cos\phi-1)^2\ge0$ that
$$
1\ge\frac{3+6\cos\phi}{(\cos\phi+2)^2}.
$$
The desired result is obtained directly upon integration of the latter inequality with respect to $\phi$ from $0$ to $\arctan x$ for $x>0$.
  \item[Solution by Konhauser]
The substitution $x=\tan y$ transforms the given inequality into $y>\frac{3\sin y}{2+\cos y}$, which is a special case of an inequality discussed on \cite[pp.~105--106]{mit-elem-ineq-book}.
\end{description}
\par
It may be worthwhile to note that the inequality~\eqref{Shafer-arctan-ineq} is not collected in the authorized monograph~\cite{kuang-3rd} and~\cite{mit}.
\par
In~\cite[pp.~288--289]{kuang-3rd}, the following inequalities for the arc tangent function are collected:
\begin{gather}
  \arctan x<\frac{2x}{1+\sqrt{1+x^2}\,},\label{kuang-coll-arctan-1}\\
  \frac{x}{1+x^2}<\arctan x<x,\label{kuang-coll-arctan-2}\\
  x-\frac{x^3}3<\arctan x<x,\label{kuang-coll-arctan-3}\\
  \frac1{2x}\ln\bigl(1+x^2\bigr)<\arctan x<(1+x)\ln(1+x),\label{ln-arctan-ineq}
\end{gather}
where $x>0$.
\par
The aim of this paper is to sharpen and generalize inequalities~\eqref{Shafer-arctan-ineq} and~\eqref{kuang-coll-arctan-1}.
\par
Our results may be stated as the following theorems.

\begin{thm}\label{Shafer-ArcTan-thm-1}
For $x>0$, let
\begin{equation}
f_a(x)=\frac{\bigl(a+\sqrt{1+x^2}\,\bigr)\arctan x}x,
\end{equation}
where $a$ is a real number.
\begin{enumerate}
\item
When $a\le-1$ or $0\le a\le\frac12$, the function $f_a(x)$ is strictly increasing on $(0,\infty)$;
\item
When $a\ge\frac2\pi$, the function $f_a(x)$ is strictly decreasing on $(0,\infty)$;
\item
When $\frac12<a<\frac2\pi$, the function $f_a(x)$ has a unique minimum on $(0,\infty)$.
\end{enumerate}
\end{thm}

As direct consequences of Theorem~\ref{Shafer-ArcTan-thm-1}, the following inequalities may be derived.

\begin{thm}\label{Shafer-ArcTan-thm-2}
For $0\le a\le\frac12$,
\begin{equation}\label{Shafer-ArcTan-thm-2-ineq-1}
\frac{(1+a)x}{a+\sqrt{1+x^2}\,}<\arctan x<\frac{(\pi/2)x}{a+\sqrt{1+x^2}\,}, \quad x>0.
\end{equation}
For $\frac12<a<\frac2\pi$,
\begin{equation}\label{Shafer-ArcTan-thm-2-ineq-2}
\frac{4a\bigl(1-a^2\bigr)x}{a+\sqrt{1+x^2}\,} <\arctan x <\frac{\max\{(\pi/2),1+a\}x}{a+\sqrt{1+x^2}\,}, \quad x>0.
\end{equation}
For $a\ge\frac2\pi$, the inequality~\eqref{Shafer-ArcTan-thm-2-ineq-1} is reversed.
\par
Moreover, the constants $1+a$ and $\frac\pi2$ in inequalities\eqref{Shafer-ArcTan-thm-2-ineq-1} and~\eqref{Shafer-ArcTan-thm-2-ineq-2} are the best possible.
\end{thm}

\section{Remarks}

Before proving our theorems, we are about to give several remarks on them.

\begin{rem}
The inequality~\eqref{Shafer-arctan-ineq} is the special case $a=\frac12$ of the left-hand side inequality in~\eqref{Shafer-ArcTan-thm-2-ineq-1}.
\end{rem}

\begin{rem}
The inequality~\eqref{kuang-coll-arctan-1} is the special case $a=1$ of the reversed version of the left hand-side inequality in~\eqref{Shafer-ArcTan-thm-2-ineq-1}.
\end{rem}

\begin{rem}
The inequality~\eqref{kuang-coll-arctan-1} is better than~\eqref{ln-arctan-ineq}. If taking $a=\frac2\pi$ in~\eqref{Shafer-ArcTan-thm-2-ineq-1}, then
\begin{equation}\label{Shafer-ArcTan-thm-2-ineq-2pi}
\frac{\pi^2x}{2+2\pi\sqrt{1+x^2}\,}<\arctan x<\frac{(\pi+2)x}{2+\pi\sqrt{1+x^2}\,}, \quad x>0.
\end{equation}
This double inequality refines corresponding ones in~\eqref{Shafer-arctan-ineq}, \eqref{kuang-coll-arctan-1}, \eqref{kuang-coll-arctan-2}, \eqref{kuang-coll-arctan-3} and \eqref{ln-arctan-ineq}.
\end{rem}

\begin{rem}
The substitution $x=\tan y$ may transform inequalities in~\eqref{Shafer-ArcTan-thm-2-ineq-1} and~\eqref{Shafer-ArcTan-thm-2-ineq-2} into some trigonometric inequalities.
\end{rem}

\begin{rem}
The approach below used in the proofs of Theorem~\ref{Shafer-ArcTan-thm-1} and Theorem~\ref{Shafer-ArcTan-thm-2} has been employed in~\cite{Carlson-Arccos.tex, Oppeheim-Sin-Cos.tex, Shafer-ArcSin.tex, Shafer-Fink-ArcSin.tex}.
\end{rem}

\section{Proofs of theorems}

Now we are in a position to prove our theorems.

\begin{proof}[Proof of Theorem~\ref{Shafer-ArcTan-thm-1}]
Direct calculation gives
\begin{align*}
f'_a(x)&=\frac{\bigl(1+x^2\bigr)\bigl(1+a\sqrt{1+x^2}\,\bigr)}{x^2\bigl(1+x^2\bigr)^{3/2}} \Biggl[\frac{x+x^3+ax\sqrt{1+x^2}\,}{\bigl(1+x^2\bigr)\bigl(1+a\sqrt{1+x^2}\,\bigr)}-\arctan x\Biggr]\\
&\triangleq\frac{\bigl(1+x^2\bigr)\bigl(1+a\sqrt{1+x^2}\,\bigr)}{x^2\bigl(1+x^2\bigr)^{3/2}}g_a(x),\\
g'_a(x)&=-\frac{x^2 \bigl(2a^2\sqrt{x^2+1}\,+a-\sqrt{x^2+1}\,\bigr)} {\bigl(x^2+1\bigr)^{3/2}\bigl(a\sqrt{x^2+1}\,+1\bigr)^2}\\
&\triangleq -\frac{x^2h_a(x)}{\bigl(x^2+1\bigr)^{3/2}\bigl(a\sqrt{x^2+1}\,+1\bigr)^2},
\end{align*}
and the function $h_a(x)$ has two zeros
$$
a_1(x)=-\frac{1+\sqrt{9+8x^2}\,}{4\sqrt{1+x^2}\,}\quad\text{and}\quad a_2(x)=\frac{-1+\sqrt{9+8x^2}\,}{4\sqrt{1+x^2}\,}.
$$
Further differentiation yields
\begin{gather*}
a_1'(x)=\frac{x\bigl(1+\sqrt{9+8x^2}\,\bigr)}{4(1+x^2)^{3/2}\sqrt{9+8x^2}\,}>0
\end{gather*}
and
\begin{gather*}
a_2'(x)=\frac{x\bigl(\sqrt{9+8x^2}\,-1\bigr)}{4(1+x^2)^{3/2}\sqrt{9+8x^2}\,}>0.
\end{gather*}
This means that the functions $a_1(x)$ and $a_2(x)$ are increasing on $(0,\infty)$. From
\begin{align*}
\lim_{x\to0^+}a_1(x)&=-1,& \lim_{x\to\infty}a_1(x)&=-\frac{\sqrt2\,}2,\\
\lim_{x\to0^+}a_2(x)&=\frac12,& \lim_{x\to\infty}a_2(x)&=\frac{\sqrt2\,}2,
\end{align*}
it follows that
\begin{enumerate}
\item
when $a\le-1$ or $a\ge\frac{\sqrt2\,}2$, the derivative $g'_a(x)$ is negative and the function $g_a(x)$ is strictly decreasing on $(0,\infty)$. From
\begin{equation}\label{g-a-limits}
\lim_{x\to0^+}g_a(x)=0\quad\text{and}\quad \lim_{x\to\infty}g_a(x)=\frac1a-\frac\pi2,
\end{equation}
it is deduced that $g_a(x)<0$ on $(0,\infty)$. Accordingly,
\begin{enumerate}
\item
when $a\le-1$, the derivative $f_a'(x)>0$ and the function $f_a(x)$ is strictly increasing on $(0,\infty)$;
\item
when $a\ge\frac{\sqrt2\,}2$, the derivative $f_a'(x)$ is negative and the function $f_a(x)$ is strictly decreasing on $(0,\infty)$.
\end{enumerate}
\item
when $\frac12\ge a\ge0$, the derivative $g'_a(x)$ is positive and the function $g_a(x)$ is increasing on $(0,\infty)$. By~\eqref{g-a-limits}, it follows that the function $g_a(x)$ is positive on $(0,\infty)$. Thus, the derivative $f_a'(x)$ is positive and the function $f_a(x)$ is strictly increasing on $(0,\infty)$.
\item
when $\frac12<a<\frac{\sqrt2\,}2$, the derivative $g'_a(x)$ has a unique zero which is a minimum of $g_a(x)$ on $(0,\infty)$. Hence, by the second limit in~\eqref{g-a-limits}, it may be deduced that
\begin{enumerate}
\item
when $\frac2\pi\le a<\frac{\sqrt2\,}2$, the function $g_a(x)$ is negative on $(0,\infty)$, so the derivative $f_a'(x)$ is also negative and the function $f_a(x)$ is strictly decreasing on $(0,\infty)$;
\item
when $\frac12<a<\frac2\pi$, the function $g_a(x)$ has a unique zero which is also a unique zero of the derivative $f_a'(x)$, and so the function $f_a(x)$ has a unique minimum of the function $f_a(x)$ on $(0,\infty)$.
\end{enumerate}
\end{enumerate}
The proof of Theorem~\ref{Shafer-ArcTan-thm-1} is complete.
\end{proof}

\begin{proof}[Proof of Theorem~\ref{Shafer-ArcTan-thm-2}]
Direct calculation yields
$$
\lim_{x\to0^+}f_a(x)=1+a \quad \text{and}\quad \lim_{x\to\infty}f_a(x)=\frac\pi2.
$$
By the increasing monotonicity in Theorem~\ref{Shafer-ArcTan-thm-1}, it follows that $1+a<f_a(x)<\frac\pi2$ for $0\le a\le\frac12$, which can be rewritten as~\eqref{Shafer-ArcTan-thm-2-ineq-1}. Similarly, the reversed version of the inequality~\eqref{Shafer-ArcTan-thm-2-ineq-1} and the right-hand side inequality in~\eqref{Shafer-ArcTan-thm-2-ineq-2} can be procured.
\par
When $\frac12<a<\frac2\pi$, the unique minimum point $x_0\in(0,\infty)$ of the function $f_a(x)$ satisfies
$$
\arctan x_0=\frac{x_0+x_0^3+ax_0\sqrt{1+x_0^2}\,}{\bigl(1+x_0^2\bigr)\bigl(1+a\sqrt{1+x_0^2}\,\bigr)},
$$
and so the minimum of $f_a(x)$ on $(0,\infty)$ is
\begin{align*}
f_a(x_0)&=\frac{x_0+x_0^3+ax_0\sqrt{1+x_0^2}\,}{\bigl(1+x_0^2\bigr)\bigl(1+a\sqrt{1+x_0^2}\,\bigr)} \cdot\frac{a+\sqrt{1+x_0^2}\,}{x_0} \\
&=\frac{\bigl(a+\sqrt{1+x_0^2}\,\bigr)\bigl(1+x_0^2+a\sqrt{1+x_0^2}\,\bigr)} {\bigl(1+x_0^2\bigr)\bigl(1+a\sqrt{1+x_0^2}\,\bigr)}\\
&=\frac{(a+u)^2} {u\bigl(1+au\bigr)},\quad u=\sqrt{1+x_0^2}\,\in(1,\infty)\\
&>4a\bigl(1-a^2\bigr),
\end{align*}
as a result, the left-hand side inequality in~\eqref{Shafer-ArcTan-thm-2-ineq-2} follows. The proof of Theorem~\ref{Shafer-ArcTan-thm-2} is complete.
\end{proof}

\end{document}